\documentclass[english,10pt]{amsart}
\usepackage[english]{babel}

\usepackage[matrix,arrow]{xy}
\xyoption{all}
\usepackage{amscd,amssymb,amsfonts,amsmath}

\usepackage{graphics}
\usepackage{epsfig}
\usepackage{mathrsfs}

\newcommand\mypagesizel{
\textwidth= 6.5in
\textheight=9in
\voffset-.55in
\hoffset -0.75in
\marginparwidth=56pt
}

\newcommand{\p}[0]{{\mathbb P}}

\newcommand{\N}{\textup{N}}

\newcommand{\NE}{\overline{\textup{NE}}}

\renewcommand{\phi}{\varphi}
\newcommand{\into}{\hookrightarrow}

\renewcommand{\le}{\leqslant}
\renewcommand{\ge}{\geqslant}

\newcommand{\sE}{\mathscr{E}}
\newcommand{\sF}{\mathscr{F}}
\newcommand{\sG}{\mathscr{G}}

\newcommand{\sL}{\mathscr{L}}

\newcommand{\sN}{\mathscr{N}}
\newcommand{\sO}{\mathscr{O}}

\mypagesizel

\newtheorem{thm}{Theorem}[section]

\newtheorem{question}[thm]{Question}
\newtheorem{lemma}[thm]{Lemma}
\newtheorem{cor}[thm]{Corollary}

\newtheorem{prop}[thm]{Proposition}

\newtheorem{conj}[thm]{Conjecture}

\newtheorem*{thm*}{Theorem}
\theoremstyle{definition}
\newtheorem{defn}[thm]{Definition}

\newtheorem{say}[thm]{}

\newtheorem{notation}[thm]{Notation}

\newtheorem{defn-thm}[thm]{Definition-Theorem} 
\newtheorem{defn-lemma}[thm]{Definition-Lemma}

\newtheorem{rem}[thm]{Remark}

\theoremstyle{remark}

\newtheorem*{not-and-def}{Notation and definitions}

\numberwithin{equation}{section}

\def\factor#1.#2.{\left. \raise 2pt\hbox{$#1$} \right/\hskip -2pt\raise -2pt\hbox{$#2$}}

\begin{document} 

\title[Regular foliations on weak Fano manifolds]{Regular foliations on weak Fano manifolds}

\author{St\'ephane \textsc{Druel}}

\address{St\'ephane Druel: Institut Fourier, UMR 5582 du
  CNRS, Universit\'e Grenoble 1, BP 74, 38402 Saint Martin
  d'H\`eres, France} 

\email{druel@ujf-grenoble.fr}

%\thanks{The author was 
%partially supported by the project \textit{CLASS} of
%Agence nationale de la recherche, under agreement 
%ANR-10-JCJC-0111}

\subjclass[2010]{37F75}

\begin{abstract}
In this paper we prove that a regular foliation on a complex weak Fano manifold is algebraically integrable.
\end{abstract}

\maketitle

{\small\tableofcontents}

\section{Introduction}

In this paper, we provide some evidence for the following conjecture.

\begin{conj}[F. Touzet]
Let $X$ be a complex projective manifold, and let $\sF$ be a regular foliation on $X$. Suppose that $X$ is rationally connected. Then
$\sF$ is algebraically integrable.
\end{conj}

The statement is a tautology in the case of curves. For surfaces, it follows from the classification of foliation by curves on surfaces (\cite{brunella_classification}). It was also known to be true if $X$ is a rational homogeneous space (see \cite{ghys_classification} and \cite{bianco_jvp}). 
%Notice that rational homogeneous spaces are Fano varieties.
Our result is the following.
Recall that a \emph{weak Fano} manifold is
a complex projective manifold $X$ such that $-K_X$ is nef and big. 
%Weak Fano manifolds are known to be rationally connected.

\begin{thm}\label{thm:main}
Let $X$ be a complex weak Fano manifold, and let $\sF\subseteq T_X$ be a 
regular foliation. Then the foliation $\sF$ is given by the fibers of a smooth morphism
$X \to Y$ onto a projective manifold.
\end{thm}

\begin{rem}In the setup of Theorem \ref{thm:main}, $Y$ is a weak Fano manifold by 
\cite[Theorem 1.1]{fujino_gongyo}.
\end{rem}

\begin{rem}
Let $n \ge 2$ be an integer, and let $\sF$ be a foliation on $\mathbb{P}^n$ induced by a general global holomorphic vector field. Then the leaf of $\sF$ through a general point in not algebraic. This shows that 
Theorem \ref{thm:main} is wrong if one drops the regularity assumption on $\sF$.
\end{rem}

In order to prove Theorem 1.1, we consider the normal bundle $\sN:=T_X/\sF$ of the foliation $\sF$. 
We show first that $\det(\sN)$ is nef. This follows from a foliated version of the bend-and-break lemma (see also Proposition \ref{prop:bend_and_break_foliation}). 

\begin{prop}\label{prop:bend_and_break_foliation2}
Let $X$ be a complex projective manifold, and let $\sF \subsetneq T_X$ be a regular foliation with normal bundle $\sN$. Let 
$C \subset X$ be a rational curve with $\det(\sN)\cdot C \neq 0$, and let $x$ be a point on $C$. 
If $- K_\sF\cdot C \ge 1$, then there exist a nonzero effective rational $1$-cycle $Z$ passing through $x$, a rational curve $C_1$, and a positive integer $m$ such that $C \equiv mC_1+Z$ and such that
$\textup{Supp}(Z)$ is tangent to $\sF$.
\end{prop}
From the base-point-free theorem, we conclude that $\det(\sN)$ is semi-ample. We then prove that the corresponding morphism $\phi\colon X \to Y$ yields a first integral for $\sF$ as follows. Let $F$ be a general fiber of $\phi$. By the adjunction formula, $F$ is a weak Fano manifold. In particular, $F$ does not carry 
differential forms. This easily implies that $F$ is tangent to $\sF$ (see Lemma \ref{lemma:first_integral}).
On the other hand, the Baum-Bott vanishing theorem yields $\dim Y \le \dim X - \textup{rank } \sF$, and hence
$\dim F = \textup{rank } \sF$, completing the proof of the claim.

\medskip

\noindent {\bf Acknowledgements.} We would like to thank 
Jorge V. \textsc{Pereira} and
Fr\'ed\'eric \textsc{Touzet} for helpful discussions.

\section{Recollection: Foliations}

In this section we recall the basic facts concerning foliations.

\subsection{Foliations}

\begin{defn}
A \emph{foliation} on a complex manifold $X$ is a coherent subsheaf $\sF\subseteq T_X$ such that
\begin{itemize}
	\item $\sF$ is closed under the Lie bracket, and
	\item $\sF$ is saturated in $T_X$. In other words, the quotient $T_X / \sF$ is torsion free.
\end{itemize}
The \emph{rank} $r$ of $\sF$ is the generic rank of $\sF$.
The \emph{codimension} of $\sF$ is defined as $q:=\dim X-r$. 
Let $X^\circ \subset X$ be the maximal open set where $\sF$ is a subbundle of $T_X$. 
We say that $\sF$ is \emph{regular} if $X^\circ=X$.

A \emph{leaf} of $\sF$ is a connected, locally closed holomorphic submanifold $L \subset X^\circ$ such that
$T_L=\sF_{|L}$. A leaf is called \emph{algebraic} if it is open in its Zariski closure.

The foliation $\sF$ is said to be \textit{algebraically
integrable} if its leaves are algebraic.
\end{defn}

\begin{defn}
Let $\sF$ be a foliation on a smooth variety $X$.
The \textit{canonical class} $K_{\sF}$ of $\sF$ is any Weil divisor on $X$ such that  $\sO_X(-K_{\sF})\cong \det(\sF)$. 
\end{defn}

\begin{say}[Foliations defined by $q$-forms] \label{q-forms}
Let $q$ and $n$ be positive integers.
Let $\sF$ be a codimension $q$ foliation on an $n$-dimensional complex manifold $X$.
The \emph{normal sheaf} of $\sF$ is $\sN:=(T_X/\sF)^{**}$.
The $q$-th wedge product of the inclusion
$\sN^*\into (\Omega^1_X)^{**}$ gives rise to a nonzero global section 
 $\omega\in H^0\big(X,\Omega^{q}_X\otimes \det(\sN)\big)$
 whose zero locus has codimension at least $2$ in $X$. 
Such $\omega$ is \emph{locally decomposable} and \emph{integrable}.
To say that $\omega$ is locally decomposable means that, 
in a neighborhood of a general point of $X$, $\omega$ decomposes as the wedge product of $q$ local $1$-forms 
$\omega=\omega_1\wedge\cdots\wedge\omega_q$.
To say that it is integrable means that for this local decomposition one has 
$d\omega_i\wedge \omega=0$ for every $i\in\{1,\ldots,q\}$. 
The integrability condition for $\omega$ is equivalent to the condition that $\sF$ 
is closed under the Lie bracket.

Conversely, let $\sL$ be a line bundle on $X$, $q\ge 1$, and 
$\omega\in H^0(X,\Omega^{q}_X\otimes \sL)$ a global section
whose zero locus has codimension at least $2$ in $X$.
Suppose that $\omega$  is locally decomposable and integrable.
Then  one defines 
a foliation of rank $r=n-q$ on $X$ as the kernel
of the morphism $T_X \to \Omega^{q-1}_X\otimes \sL$ given by the contraction with $\omega$. 
These constructions are inverse of each other. 
 \end{say}

We will need the following easy observation.

\begin{lemma}\label{lemma:first_integral}
Let $q$ be a positive integer, and 
let $\sF$ be a codimension $q$ foliation on a complex projective manifold $X$. Let 
$\phi\colon X \to Y$ be a surjective morphism with connected fibers onto a normal projective variety 
$Y$, with general fiber $F$. Set $\sN:=T_X/\sN$ and $\sL := \det (\sN)$.
Suppose that $\sL_{|F} \sim 0$ and that $h^0(F,\Omega_F^i)=0$ for all $1 \le i \le \dim F$. Then $F$ is tangent to $\sF$. In particular, we have $\dim Y \ge q$.
\end{lemma}

\begin{proof}
Let $\omega\in H^0(X,\Omega_X^q\otimes\sL)$ be a twisted $q$-form defining $\sF$ (see \ref{q-forms}).
%Notice that  
%$\sL_{|F}\cong \sO_F$ since $h^1(X,\sO_X)=h^0(X,\Omega_X^1)=0$ by assumption.
The short exact sequence
$$0 \to \sN_{F/X}^*\cong\sO_F^{\oplus \dim X - \dim Y} \to {\Omega_X^1}_{|F} \to \Omega^1_F \to 0 $$
yields a filtration
$$\{0\}=\sE_{q+1}\subseteq \sE_q \subseteq \cdots \subseteq \sE_0={\Omega_X^q}_{|F}$$
with 
$$\sE_i/\sE_{i+1}\cong \wedge^i\big(\sN_{F/X}^*\big) \otimes \Omega^{q-i}_F.$$
Since $h^0(F,\Omega_F^{q-i}) = 0$ for all $0\le i\le q-1$, 
we conclude that 
$$\omega_{|F}\in H^0(F,\sE_q)=H^0(F,\wedge^q(\sN_{F/X}^*)\subset H^0(F,{\Omega_X^q}_{|F}).$$
This implies that $q \ge \dim Y$ and that 
$\sN^*_{|F^\circ} \subset {\sN_{F/X}^*}_{|F^\circ}\subset {\Omega_X^1}_{|F^\circ}$, where
$X^\circ \subset X$ denotes the maximal open set where $\sF$ is a subbundle of $T_X$, and
$F^\circ := F \cap X^\circ$. 
Thus $T_{F^\circ} \subset \sF_{|F^\circ}$, proving the lemma. 
\end{proof}

\subsection{Bott (partial) connection}

\begin{say}\label{partial_connection}
Let $X$ be a complex manifold, let $\sF\subset T_X$ be a regular codimension $q$ foliation 
with $0 < q < \dim X$, and set $\sN=T_X/\sF$. Let $p\colon T_X\to \sN$ denotes the natural projection. For sections $U$ of $\sN$, $T$ of $T_X$, and $V$ of $\sF$ over some open subset of $X$ with 
$U=p(T)$, set $D_V U=p([V,U])$. This expression is well-defined,
$\sO_X$-linear in $V$ and satisfies the Leibnitz rule 
$D_V(fU)=fD_V U+(Vf)U$ so that $D$ is an $\sF$-connection on $\sN$
(see \cite{baum_bott70}). 
\end{say}

\begin{lemma}\label{lemma:normal_bundle_restricted_leaf}
Let $X$ be a complex manifold, and let $\sF\subsetneq T_X$ be a regular foliation 
with normal bundle $\sN=T_X/\sF$. Let $f\colon Z \to X$ be a compact manifold, and 
suppose that $f(Z)$ is tangent to $\sF$. Then $f^*\sN$ admits a holomorphic flat connection. In particular,  characteristic classes of $f^*\sN$ vanish.
\end{lemma}

\begin{proof}
This follows from \ref{partial_connection} and \cite{atiyah57}.
\end{proof}

\section{Deformations of a morphism along a foliation}

In this section, we provide a technical tool for the proof of the main result (see Corollary \ref{cor:extremal_rays}).

\begin{say} Let $Z$, $Y$ and $X$ be normal complex projective varietes, and let $g \colon Z \to X$ be a morphism. Let $\textup{Hom}\big(Y,X\big)$ denotes the space of morphisms $f \colon Y \to X$, and let
$\textup{Hom}\big(Y,X;g\big)\subset \textup{Hom}\big(Y,X\big)$ denotes the Zariski closed subspace parametrizing
morphisms $f \colon Y \to X$ such that $f_{|Z}=g$ (see \cite[Proposition 1]{mori}).

\medskip

Suppose now that $Z$, $Y$ and $X$ are complex projective manifolds, and consider a codimension $q$ regular foliation 
$\sF \subseteq T_X$ on $X$ with $0<q<\dim X$. 
Pick $[f] \in \textup{Hom}\big(Y,X\big)$. 
Let $\textup{Def}\big([f],\sF\big)$
denotes the germ of analytic space parametrizing
small
deformations of $[f]$ along $\sF$. It is constructed as follows (see \cite[Sect. 6]{miyaoka87}, or
\cite[Corollary 5.6]{kebekus_kousidis_lohmann}). Choose an open cover $(U_i)_{i\in I}$ of $X$ with respect to the Euclidean topology such that, for each $i \in I$, $\sF_{|U_i}$ is induced by a holomorphic submersion 
$\phi_i\colon U_i \to W_i$ of complex analytic spaces. 
Let $(V_j)_{j\in J}$ be a finite open cover of $Y$. By replacing 
$(V_j)_{j\in J}$ with a refinement, we may assume that, for each $j\in J$,
there exist $i_j\in I$ and an open neighborhood $H_j$ of $[f]$ such that
$h(y) \in U_{i_j}$ for each $[h] \in H_j$ and each $y \in V_j$.
Let $\textup{Def}\big([f],\sF\big)$ be the connected component of the intersection 
$$\bigcap_{j\in J} \Big\{[h]\in H_j\quad |
\quad \phi_{i_j}\circ (h_{|V_j})=\phi_{i_j}\circ ({f}_{|V_j})\Big\} $$
which contains $[f]$. Notice that $\textup{Def}\big([f],\sF\big)$ is a locally closed (possibly nonreduced)
analytic subset. 
Set $$\textup{Def}\big([f],\sF;g\big)=\textup{Def}\big([f],\sF\big) \cap \textup{Hom}\big(Y,X;g\big).$$
\end{say}

\begin{rem} Let $\phi\colon X \to Y$ be a surjective morphism with connected fibers of projective manifolds, 
let $Z$ be a projective manifold, and let 
$f \colon Z \to X$ be a morphism. Let $\sF$ be the foliation on $X$ given by the fibers of $\phi$. Recall that the space of deformations of $[f]$ over $Y$ are parametrized by the fiber 
$\textup{Hom}_Y\big(Z,X\big)$ of $[\phi\circ f]$
under the map
$$\textup{Hom}\big(Z,X\big) \to \textup{Hom}\big(Z,Y\big).$$
Suppose that $\sF$ is regular. Then we have an embedding 
$\Big(\textup{Def}\big([f],\sF\big),[f]\Big) \subseteq \Big(\textup{Hom}_Y\big(Z,X\big),[f]\Big)$
of pointed analytic spaces but they are not isomorphic in general. Indeed, suppose that 
$\dim Y =1$. Let $y$ be a point on $Y$, and set $F:=\phi^{-1}(y)_{\textup{red}}$. Suppose that 
the multiplicity $m$ of $F$ is $> 1$. Let $Z$ be a reduced point $\{z\}$, and suppose that $f(z) \subset F$.
Then $\Big(\textup{Def}\big([f],\sF\big),[f]\Big) \cong \big(F,z\big) \cong \Big(\textup{Hom}_Y\big(Z,X\big)_{\textup{red}},[f]\Big)$
while $\Big(\textup{Hom}_Y\big(Z,X\big),[f]\Big)\cong \big(\phi^{-1}(y),z\big)$.
\end{rem}

The following observation will prove to be crucial.
It is due to Loray, Pereira and Touzet (see proof of \cite[Proposition 6.12]{loray_pereira_touzet}).
%The following notation will be used.

\begin{notation}
Let $(A,a)$ be a pointed analytic space. We denote by $\widehat{A}$ the formal completion of $A$ at $a$.
Given a morphism of pointed analytic spaces $\lambda\colon (A,a) \to (B,b)$, we denote by 
$\widehat{\lambda}\colon \widehat{A} \to \widehat{B}$ the induced morphism of formal analytic spaces.
\end{notation}

\begin{lemma}\label{lemma:algebraicity}
Let $Y$ and $X$ be complex projective manifolds, and let $\sF \subseteq T_X$ be a regular foliation. Let $f \colon Y \to X$ be a morphism, and let $y$ be a point on $Y$. 
Then the Zariski closure $T$ 
of ${\textup{Def}\big([f],\sF;f_{|\{y\}}\big)}_{\textup{red}}$
in ${\textup{Hom}\big(Y,X;f_{|\{y\}}\big)}_{\textup{red}}$
parametrizes deformations of $[f]$ along $\sF$. In other words, for each $y' \in Y$, 
$ev\big(T \times \{y'\}\big)$ is tangent to $\sF$, where 
$ev \colon \textup{Hom}\big(Y,X;f_{|\{y\}}\big) \times Y \to X$ denotes the 
evaluation morphism. 
\end{lemma}

\begin{proof}
Set $x:=f(y)$, and let $U$ be an open neighborhood of $x$ in $X$ 
with respect to the Euclidean topology
such that $\sF_{|U}$ is induced by a submersion
$\phi \colon U \to W$ of complex analytic spaces. 
Let $\widehat{T}$
be the connected component containing $[f]$ of the \textit{Zariski closed} subset 
$$\Big\{[h]\in {\textup{Hom}\big(Y,X;f_{|\{y\}}\big)}_{\textup{red}}\quad |\quad 
\widehat{\phi \circ h}=\widehat{\phi \circ f}\colon \widehat{Y}\to \widehat{W}   
\Big\} \subset {\textup{Hom}\big(Y,X;f_{|\{y\}}\big)}_{\textup{red}}.$$
Notice that $T \subset \widehat{T}$.
Let $[h]\in \widehat{T}$, and consider an open neighborhood 
$V$ of $y$ and an open neighborhood $H$ of $[h]$ in $\widehat{T}$
(with respect to the Euclidean topology)
such that for each $[h'] \in H$ and each $y'\in V$, we have $h'(y')\in U$.
If $[h']\in H$, then $$\phi\circ ({h'}_{|V})=\phi\circ ({h}_{|V})\colon V \to W \quad\textup{since}\quad
\widehat{\phi \circ h'}=\widehat{\phi \circ f}=\widehat{\phi \circ h}\colon \widehat{Y}\to \widehat{W}.$$ 
This implies that
$ev\big(H\times\{y'\}\big)$ is tangent to $\sF$ for each $y' \in V$, and hence so is 
$ev\big(\widehat{T}\times\{y'\}\big)$. 
Since the set of points $y'\in Y$ such that $ev\big(\widehat{T}\times\{y'\}\big)$ is tangent to $\sF$ is Zariski closed in
$Y$, 
we conclude that $ev\big(\widehat{T}\times\{y'\}\big)$ is tangent to $\sF$ for any $y' \in Y$. This proves the lemma.
\end{proof}

\begin{rem}
One might ask whether Lemma \ref{lemma:algebraicity} holds for a larger class of foliations. What we
actually proved is the following. If $\sF$ is induced on an open neighborhood $U$ of $y$ (with respect to the
Euclidean topology) by a holomorphic map $U \to V$ of complex spaces, then the conclusion of Lemma \ref{lemma:algebraicity} holds.
 \end{rem}

The following lemma provides a lower bound for the dimension of $\textup{Def}([f],\sF;f_{|B})$ at a point $[f]$, thereby allowing us in certain situations to produce many deformations of $f$ (see Proposition \ref{prop:bend_and_break_foliation2}). 

\begin{lemma}\label{lemma:dimension}
Let $X$ be a complex projective manifold, and let $\sF \subseteq T_X$ be a regular rank $r$ foliation on $X$. Let $f \colon C \to X$ be a smooth curve, and let $B$ be a finite subscheme of $C$. Then 
$$\dim_{[f]} \textup{Def}\big([f],\sF;f_{|B}\big) \ge -K_\sF \cdot f_*C + \big(1-g(C)-\ell(B)\big)\cdot r.$$
\end{lemma}

\begin{proof}Let $(\sO,\mathfrak{m})$ be local ring of the germ of analytic space 
$\textup{Def}\big([f],\sF;f_{|B}\big)$ at $[f]$, and let $\widehat{\sO}$ be its $\mathfrak{m}$-adic 
completion.
Then $\widehat{\sO}$ pro-represents the functor of infinitesimal deformations of $[f]$ along $\sF$ with fixed subscheme $B$. We refer to \cite[Section 6]{miyaoka87} for the definition of this functor. The lemma then follows from \cite[Theorem 6.2]{miyaoka87} (see also \cite[Corollary 6.6]{miyaoka87}).
\end{proof}

The proof of Proposition \ref{prop:bend_and_break_foliation} below is very similar to that of 
\cite[Proposition 3.1]{debarre} (see also \cite[Proposition 6.13]{loray_pereira_touzet}), and so we leave some easy details to the reader.

\begin{prop}\label{prop:bend_and_break_foliation}
Let $X$ be a complex projective manifold, and let $\sF \subseteq T_X$ be a regular foliation. Let $f \colon C \to X$ be a smooth complete curve, and let $c$ be a point on $C$. 
If $C\cong \mathbb{P}^1$, suppose that $f(C)$ is transverse to $\sF$ at a general point on $f(C)$.
Suppose furthermore that $\dim_{[f]} \textup{Def}\big([f],\sF;f_{|\{c\}}\big) \ge 1$. There exist a morphism 
$g\colon C \to X$, a nonzero effective rational $1$-cycle $Z$ on $X$ passing through $f(c)$ such that
$f_*C\equiv g_*C+Z$ and such that $\textup{Supp}(Z)$ is tangent to $\sF$.
\end{prop}

\begin{proof}
Denote by 
${\overline{\textup{Def}\big([f],\sF;f_{|\{c\}}\big)}}_{\textup{red}} 
\subset {\textup{Hom}\big(Y,X;f_{|\{c\}}\big)}_{\textup{red}}$ the Zariski closure of 
${\textup{Def}\big([f],\sF;f_{|\{c\}}\big)}_{\textup{red}}$.
Let $T \to {\overline{\textup{Def}\big([f],\sF;f_{|\{c\}}\big)}}_{\textup{red}}$ be the normalization of a $1$-dimensional subvariety
passing through $[f]$, and let $\overline{T}$ be a smooth compactification. 
Let $e\colon S \overset{\varepsilon}{\to} C \times\overline{T} \overset{ev}{\dashrightarrow} X$
be a resolution of the indeterminacies of the rational map
$ev\colon C \times \overline{T} \dashrightarrow X$ coming from $T \to \textup{Hom}\big(C,X;f_{|\{c\}}\big)$, where
$\epsilon \colon S \to C \times \overline{T}$ is obtained by blowing-up points.
From the rigidity lemma, we conclude that there exists a point $t_0 \in \overline{T}$ such that $ev$ is not defined at $(c,t_0)$. The fiber of $t_0$ under the projection $S \to \overline{T}$ is the union of the strict transform
of $C \times \{t_0\}$ and a (connected) exceptional rational $1$-cycle $E$ which is not entirely contracted by $e$ and meets the strict transform of $\{c\}\times \overline{T}$. Since the latter is contracted by $e$ to the point
$f(c)$, the rational $1$-cycle $Z:=e_*E$ passes throuh $f(c)$.

By Lemma \ref{lemma:algebraicity}, ${\overline{\textup{Def}\big([f],\sF;f_{|\{c\}}\big)}}_{\textup{red}}$
parametrizes deformations of $[f]$ along $\sF$. 
Therefore, if $C$ is transverse to $\sF$ at a general point on $C$, $\textup{Aut}(C,c)\cdot [f]$ and 
${\overline{\textup{Def}\big([f],\sF;f_{|\{c\}}\big)}}_{\textup{red}}$ intersect at finitely many points in $\textup{Hom}\big(C,X;f_{|\{c\}}\big)$. If $C$ is irrational, then the orbit
$\textup{Aut}(C,c)\cdot [f]$ is finite because the 
group $\textup{Aut}(C,c)$ is. In either case, we conclude that
$\dim e(S)=2$. 

Let $\sG\subseteq T_{C \times \overline{T}}$ be the foliation on $C \times \overline{T}$ induced by $ev^*\sF \cap T_{C\times\overline{T}}$, and set $\sG_S:=\varepsilon^{-1}(\sG)$.
If $C$ is tangent to $\sF$, then $\sG=T_{C\times\overline{T}}$ (and hence $\sG_S = T_S$).
If $C$ is transverse to $\sF$ at a general point on $C$, then 
$\sG$ is induced by the projection $C\times \overline{T} \to C$.
In either case, any $\varepsilon$-exceptional curve is tangent to $\sG_S$. Hence 
$\textup{Supp}(Z)$ is tangent to $\sF$. This completes the proof of the proposition.
\end{proof}

\begin{proof}[Proof of Proposition \ref{prop:bend_and_break_foliation2}]
Let $X$ be a complex projective manifold, and let $\sF \subsetneq T_X$ be a regular foliation with normal bundle $\sN$. Let 
$C \subset X$ be a rational curve with $\det(\sN)\cdot C \neq 0$, and let $x$ be a point on $C$. 
Suppose that $- K_\sF\cdot C \ge 1$.
Let $f\colon \p^1 \to C \subset X$ be the normalization morphism, and let $p\in\p^1$ such that $f(p)=x$. 
Notice that $C$ is tranverse to $\sF$ at a general point on $C$ by 
Lemma \ref{lemma:normal_bundle_restricted_leaf}.
By Lemma \ref{lemma:dimension}, we have 
$$\dim_{[f]} \textup{Def}\big([f],\sF;f_{|\{p\}}\big) \ge -K_\sF\cdot C \ge 1$$
so that Proposition \ref{prop:bend_and_break_foliation} applies. There exist a morphism 
$g\colon \p^1 \to X$ and a nonzero effective rational $1$-cycle $Z$ on $X$ such that
$f_*\p^1\equiv g_*\p^1+Z$, and such that $\textup{Supp}(Z)$ is tangent to $\sF$.
From Lemma \ref{lemma:normal_bundle_restricted_leaf} again, we deduce that 
$\det(\sN)\cdot Z=0$. Thus
$$0\neq \det(\sN)\cdot f_*\p^1 = \det(\sN)\cdot g_*\p^1 +\det(\sN)\cdot Z = \det(\sN)\cdot g_*\p^1.$$
In particular, $g$ is a nonconstant morphism. 
Set $C_1:=g(\mathbb{P}^1)$ and $m:=\deg(g)$. Then $C\equiv mC_1+Z$,
completing the proof of the proposition.
\end{proof}

We now provide a technical tool for the proof of the main result.

\begin{cor}\label{cor:conormal_nef}
Let $X$ be a complex projective manifold, and let $\sF \subsetneq T_X$ be a regular foliation with normal bundle $\sN$. 
Suppose that $-K_X$ is nef. If $C \subset X$ is a rational curve, then $\det(\sN) \cdot C \ge 0$.
\end{cor}

\begin{proof}Set $\sL:=\det(\sN)$, and pick an ample divisor $H$ on $X$.
We argue by contradiction, and assume that $\sL\cdot C<0$ for some rational curve $C$. We have
$-K_\sF\cdot C = -K_X\cdot C -\sL\cdot C \ge 1$
so that Proposition \ref{prop:bend_and_break_foliation2} applies.
There exist a nonzero effective rational $1$-cycle $Z$, a rational curve $C_1$, and a positive integer $m$ such that $C \equiv mC_1+Z$ and such that
$\textup{Supp}(Z)$ is tangent to $\sF$. 
Notice that $H\cdot C_1< H\cdot C$.
By Lemma \ref{lemma:normal_bundle_restricted_leaf}, we have
$$\sL\cdot C_1=\frac{1}{m}\sL\cdot(mC_1+Z)=\frac{1}{m}\sL\cdot C<0.$$
This construction yields an infinite sequence of rational curves on $X$ with 
decreasing $H$-degrees. This is absurd and the corollary is proved.
\end{proof}

Let $X$ be a complex projective manifold and consider the finite dimensional $\mathbb{R}$-vector space 
$$\N_1(X)=\big(\{1-\text{cycles}\}/\equiv\big)\otimes\mathbb{R},$$
where $\equiv$ denotes numerical equivalence. Recall that the \emph{Mori cone} of $X$ is the closure 
$\NE(X)\subset \N_1(X)$ of the cone spanned by classes of effective curves. 
We believe that the following result will be useful when considering regular foliations on arbitrary projective manifold. Its proof is similar to that of Corollary \ref{cor:conormal_nef} above.

\begin{cor}\label{cor:extremal_rays}
Let $X$ be a complex projective manifold, and let $\sF \subsetneq T_X$ be a regular foliation with normal bundle $\sN$. Let 
$C \subset X$ be a rational curve with $\det(\sN)\cdot C \neq 0$. 
If $[C]\in \NE(X)$ generates an extremal ray,
then $K_\sF\cdot C \ge 0$.
\end{cor}

\begin{proof}
Pick an ample divisor $H$ on $X$.
Let us assume to the contrary that $- K_\sF\cdot C \ge 1$. 
 By Proposition \ref{prop:bend_and_break_foliation2},
$C$ is numerically equivalent to a connected nonintegral effective rational $1$-cycle.
Thus, there exists a rational curve $C_1$ on $X$ with $[C_1]\in \mathbb{R}^+[C]$ and
such that $H \cdot C_1< H\cdot C$. Since $[C_1]\in \mathbb{R}^+[C]$, we must have $-K_\sF\cdot C_1 \ge 1$. 
This construction yields an infinite sequence of rational curves on $X$ with 
decreasing $H$-degrees. This is absurd, proving the corollary.
\end{proof}

\section{Proof of Theorem \ref{thm:main}}

We are now in position to prove our main result. 

\begin{proof}[Proof of Theorem \ref{thm:main}]
Set $\sN=T_X/\sF$, and denote by $q$ its rank.
Suppose that $0 <q<\dim X$, and set $\sL=\det(\sN)$.

By the cone theorem, there exist finitely many rational curves $C_1,\ldots,C_m$
such that $$\NE(X)=\mathbb{R}^+[C_1]+\cdots +\mathbb{R}^+[C_m]$$
where the $\mathbb{R}^+[C_i]$ are the extremal rays of $\NE(X)$ (\cite[Theorem 3.7]{kollar_mori}). 
By Corollary \ref{cor:conormal_nef}, $\sL\cdot C_i \ge 0$ for any $1\le i\le m$, and thus $\sL$ is nef.
By the base-point-free theorem (see \cite[Theorem 3.3]{kollar_mori}), the line bundle $\sL^{\otimes m}$ is globally generated for all integers $m$ sufficiently large. Let $\phi \colon X \to Y$ be the induced morphism. 

We will show that $\sF$ is induced by $\phi$. By \cite[Corollary 3.4]{baum_bott70}, we have $\sL^{q+1} \equiv 0$, and hence
$\dim Y \le q$. 
Let $F$ be a general fiber of $\phi$. Notice that $F$ is a smooth projective variety with $-K_F = (-K_X)_{|F}$ nef and big by the adjunction formula, and that $\sL_{|F}\equiv 0$. 
By \cite{zhang_rcc}, $F$ is simply connected and
$h^0(F,\Omega_F^{i}) = 0$ for all $1\le i\le \dim F$, 
so that
Lemma \ref{lemma:first_integral} applies. We have $\dim Y \ge q$, and
$F$ is tangent to $\sF$. This in turn implies that $\dim Y = q$, and that $\sF$ is induced by $\phi$. 
By Lemma \ref{lemma:smoothness} below, we infer that $\phi$ is a smooth morphism,
completing the proof of the theorem.
\end{proof}

\begin{lemma}\label{lemma:smoothness}
Let $X$ be a complex projective manifold, and let $\phi\colon X \to Y$ be a surjective morphism with connected fibers onto a normal projective variety $Y$. Suppose that $-K_X$ is $\phi$-nef and $\phi$-big. Suppose furthermore that the foliation 
$\sF$ on $X$ induced by $\phi$ is regular. Then $\phi$ is a smooth morphism.
\end{lemma}

\begin{proof}Pick $x \in X$, and set 
$y:=\phi(x)$ and
$F_0:=\phi^{-1}(y)_{\textup{red}}$.  
By \cite[Proposition 2.5]{hwang_viehweg}, $F_0$ has finite holonomy group $G$.
By the holomorphic version of Reeb stability theorem (see \cite[Theorem 2.4]{hwang_viehweg}), there exist a saturated open neighborhood $U$ of $F_0$ in $X$ with respect to the Euclidean topology, a (local) transversal section $S$ at $x$ with a $G$-action, an unramified Galois cover $\widehat{U} \to U$ with group $G$, a smooth proper $G$-equivariant morphism $\widehat{U} \to S$, an isomorphism $S/G \cong \phi(U)$, and a commutative diagram:

\centerline{
\xymatrix{
\widehat{U} \ar[r]^{p}\ar[d]_{\widehat{\phi}} & U \ar[d]^{\phi}, \\
 S \ar[r]_-{q} & S/G\cong \phi(U).\\
}
}

Recall that $G$ is given by the holonomy representation
$$\pi_1(F_0,x) \to \textup{Diff}(S,x).$$
Set $\widehat{F}_0:=\widehat{\phi}^{-1}(x)_{\textup{red}}$, and consider a general fiber $\widehat{F}$ of
$\widehat{\phi}$.
Notice that $-K_{\widehat{U}} \cong -p^*K_U$ is $\widehat{\phi}$-nef and $\widehat{\phi}$-big. It follows that
$-K_{\widehat{F}}$ is nef and big. Since 
$K_{\widehat{F}_0}^{\dim \widehat{F}_0} =  K_{\widehat{F}}^{\dim \widehat{F}}$, we infer that
$-K_{\widehat{F}_0}$ is nef and big as well. Since the restriction of $q$ to $\widehat{F}_0$ induces an \'etale morphism $q_{|\widehat{F}_0} \colon \widehat{F}_0 \to F_0$ of projective manifolds, we conclude that 
$-K_{F_0}$ is also nef and big. By \cite{zhang_rcc}, we must have $\pi_1(F_0,x)=\{1\}$. Therefore, the holonomy 
group $G$ is trivial, and $\phi$ is a smooth morphism. This proves the lemma. 
\end{proof}

\begin{question}
Let $X$ be a complex projective manifold, and let $\sF$ be a regular foliation on $X$. Suppose that $h^1(X,\sO_X)=0$, and that $-K_X$ is nef.
Is $\sF$ algebraically integrable?
\end{question}

\providecommand{\bysame}{\leavevmode\hbox to3em{\hrulefill}\thinspace}
\providecommand{\MR}{\relax\ifhmode\unskip\space\fi MR }
\providecommand{\MRhref}[2]{%
  \href{http://www.ams.org/mathscinet-getitem?mr=#1}{#2}
}
\providecommand{\href}[2]{#2}

%\bibliographystyle{amsalpha}
%\bibliography{foliation}

\end{document}